\documentclass[reqno, a4paper]{amsart}

\usepackage{amsfonts}
\usepackage{amsmath}
\usepackage{amssymb}
\usepackage{amsthm}
\usepackage[backend=biber, giveninits=true, maxbibnames=99]{biblatex}
\usepackage{enumitem}
\usepackage{mathtools}
\usepackage[hidelinks, pdfusetitle]{hyperref}
\usepackage[capitalise]{cleveref}
\usepackage{crossreftools}
\usepackage{tikz}
\usetikzlibrary{decorations.markings}

\DeclarePairedDelimiterX{\ha}[2]{\{}{\}}{#1 \,\delimsize|\, #2}
\DeclarePairedDelimiter{\zj}{(}{)}
\DeclarePairedDelimiter{\nr}{\|}{\|}

\DeclarePairedDelimiterX{\sk}[2]{\langle}{\rangle}{#1, #2}

\DeclareMathOperator*{\argmin}{arg\,min}
\DeclareMathOperator{\prox}{prox}
\DeclareMathOperator{\id}{id}
\DeclareMathOperator{\conv}{conv}
\DeclareMathOperator{\Int}{int}
\DeclareMathOperator{\Lip}{Lip}
\DeclareMathOperator{\graph}{graph}

\newcommand{\R}{\mathbb R}

\theoremstyle{plain}
\newtheorem{Thm}{Theorem}
\newtheorem{Prop}[Thm]{Proposition}
\newtheorem{Cor}[Thm]{Corollary}
\newtheorem{Lemma}[Thm]{Lemma}

\theoremstyle{definition}
\newtheorem{Def}[Thm]{Definition}
\newtheorem{Rem}[Thm]{Remark}
\newtheorem{Question}[Thm]{Question}

\numberwithin{Thm}{section}

\crefname{Prop}{Proposition}{Propositions}
\crefname{Thm}{Theorem}{Theorems}
\crefname{equation}{}{}

\setlist[enumerate]{label=(\arabic*)}

\renewbibmacro{in:}{}
\DeclareFieldFormat{pages}{#1}
\DeclareFieldFormat[article]{volume}{\mkbibbold{#1}}
\DeclareFieldFormat[article, inbook]{title}{\mkbibitalic{#1}}
\DeclareFieldFormat{booktitle}{#1}
\DeclareFieldFormat{journaltitle}{#1\isdot}

\title{On a question of Kolmogorov}
\author{Attila Gáspár}
\address{Institute of Mathematics, Eötvös Loránd University,
Pázmány Péter s. 1/C, 1117 Budapest, Hungary}
\email{gsprati99@gmail.com}

\subjclass[2020]{26A16, 28A12} 

\begin{document}

\begin{abstract}
Kolmogorov asked the following question: can every bounded measurable set in the plane be mapped onto a polygon by a 1-Lipschitz map with arbitrarily small measure loss? The answer is negative in general, however, the case of compact sets is still open. We present an equivalent form of the question for compact sets. Furthermore, we give a positive answer to Kolmogorov's question for specific classes of sets, most importantly, for planar sets with tube-null boundary. In particular, we show that the Sierpiński carpet can be mapped into the union of finitely many line segments by a 1-Lipschitz map with arbitrarily small displacements, answering a question of Balka, Elekes and Máthé.
\end{abstract}

\maketitle

\section{Introduction}

The following question was proposed by Kolmogorov in 1932
(see \cite{KQuestion}):

\begin{Question}[Kolmogorov]
    Let $A\subseteq \R^2$ be a measurable set such that its Lebesgue measure $\lambda(A)$ is finite. Is it true that for every $\varepsilon > 0$, there exists a 1-Lipschitz map $f$ such that $f(A)$ is a polygon and $\lambda(f(A)) \ge \lambda(A) - \varepsilon$?
\end{Question}

It was shown by Balka, Elekes and Máthé \cite{BEM} that the answer is negative in general, in particular, there is a bounded, open and simply connected counterexample $A$. A crucial property of their construction is that $\lambda\zj[\big]{\overline A} > \lambda(A)$, from which they proved that the image $f(A)$ cannot even be Jordan measurable if $\varepsilon$ is small enough. However, this construction does not work if $A$ is restricted to be a compact set. There is currently no known counterexample to Kolmogorov's question with compact $A$.

Our main goal is to give a positive result for Kolmogorov's question under some additional assumptions on $A$. We will also consider a higher dimensional analogue of Kolmogorov's problem.

\begin{Def}
    A set $A\subseteq \R^d$ is a \emph{polyhedron} if its boundary $\partial A$ can be covered by finitely many hyperplanes.
\end{Def}

For bounded sets $A\subseteq \R^2$, this is equivalent to the definition of polygon used in \cite{BEM}.
\begin{Def}
    We call a bounded measurable set $A\subseteq \R^d$ \emph{measure Kolmogorov} if for every $\varepsilon > 0$, there is a 1-Lipschitz map $f : A \to \R^d$ such that $f(A)$ is a polyhedron and $\lambda(f(A)) > \lambda(A) - \varepsilon$.
\end{Def}
Using these definitions, the question can be rephrased as asking which sets $A\subseteq \R^d$ are measure Kolmogorov.

A potential application of this higher dimensional variant comes from a related problem of Laczkovich \cite{Laczkovich}:
\begin{Question}[Laczkovich]
    Let $A \subseteq \R^d$ be a measurable set of positive Lebesgue measure. Is there a Lipschitz map that maps $A$ onto the unit cube?
\end{Question}
For $d=2$, this was proved by Preiss \cite{Preiss}, and subsequently strengthened by Matoušek \cite{Matousek}. An alternative proof for the planar case was found by P. Jones using an earlier result of Uy \cite{Uy} (see \cite{ACsP}). The case $d \ge 3$ is still wide open. Using the regularity of the Lebesgue measure, it can be seen that if every compact set in $\R^d$ is measure Kolmogorov, then the answer to Laczkovich's question is positive in $\R^d$.

We will also study a variant of the problem where small measure loss is replaced by small displacements:
\begin{Def}
    We call a bounded set $A\subseteq \R^d$ \emph{distance Kolmogorov} if for every $\varepsilon > 0$, there is a 1-Lipschitz map $f : A \to \R^d$ such that $f(A)$ is a polyhedron and $|f(x) - x| \le \varepsilon$ for every $x\in A$.
\end{Def}
The following implication was proved in \cite{BEM}:
\begin{Thm}[Balka, Elekes, Máthé] \label{measure_to_dist}
    Suppose that every compact set in $\R^d$ ($d \ge 2$) is measure Kolmogorov. Then every compact set in $\R^d$ is distance Kolmogorov.
\end{Thm}
Although they only considered compact Lebesgue null sets for the distance Kolmogorov property, their proof works for this more general theorem with minor modifications.

Note that it is not known whether measure Kolmogorov implies distance Kolmogorov for a specific compact set $K$. For example, if $\lambda(K)=0$, then $K$ is vacuously measure Kolmogorov, on the other hand, the question of whether every such $K$ is distance Kolmogorov is open.

The reverse implication, however, holds even for specific sets. We will prove the following theorem in \cref{sec:equiv}:

\begin{Thm} \label{dist_to_measure}
    Every measurable distance Kolmogorov set is measure Kolmogorov.
\end{Thm}

Combining \cref{measure_to_dist,dist_to_measure} yields the following equivalence:

\begin{Cor}
    For every $d\ge 2$, the following are equivalent:
    \begin{enumerate}
        \item Every compact set $K\subseteq \R^d$ is measure Kolmogorov.
        \item Every compact set $K\subseteq \R^d$ is distance Kolmogorov.
    \end{enumerate}
\end{Cor}

In \cref{sec:gstrip}, we will introduce the concept of \emph{generalized strips} and an associated outer measure $\gamma$. Our key result is the following:

\begin{Thm} \label{gstrip_main}
    Let $A\subseteq \R^d$ be a bounded set, and suppose that $\gamma(\partial A)=0$, where $\partial A$ is the boundary of $A$. Then $A$ is both distance and measure Kolmogorov.
\end{Thm}

We will also give several applications of \cref{gstrip_main}.

We call a set $A\subseteq \R^d$ \emph{strip-null} if it can be covered by countably many strips of arbitrarily small total width. Here, a strip of width $w$ is the (closed) set of points between two parallel hyperplanes of distance $w$. When $d = 2$, strip-nullity coincides with \emph{tube-nullity}. The concept of tube-null sets was first used in Fourier analysis by Carbery, Soria and Vargas \cite{CSV}: they showed that every tube-null set is a set of divergence for the localization problem. Some examples of tube-null sets in the plane are sets with $\sigma$-finite 1-dimensional Hausdorff measure \cite[Proposition 8]{CSV}, the Koch curve \cite{KochTN} and the Sierpiński carpet \cite{SierpinskiTN}.

In higher dimensions, a set in $\R^d$ is tube-null if it can be covered by tubular neighborhoods of lines with arbitrarily small total cross-sectional area. This is strictly weaker than the strip-null property. For example, the sphere $S^2\subseteq \R^3$ is tube-null (by \cite[Proposition 8]{CSV}) but not strip-null, since its intersection with a strip of width $w$ has area at most $2\pi w$. Nevertheless, it was shown in \cite[Theorem 1.1]{SierpinskiTN} that if a closed set $K\subsetneq [0, 1]^d$ is $\times N$ invariant (that is, invariant under the map $(x_1, \dots, x_d) \mapsto (N x_1 \operatorname{mod} 1, \dots, N x_d \operatorname{mod} 1)$), then $K$ is strip-null.

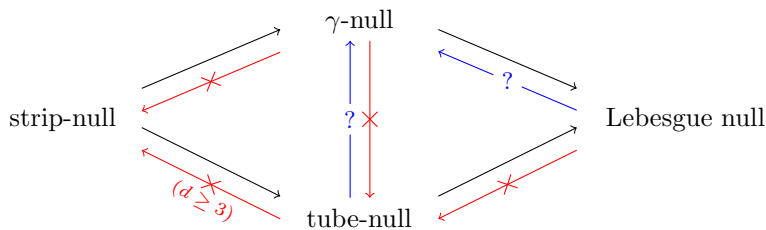
\begin{figure}
    \begin{tikzpicture}[scale=1.3,
            -x>/.style = {
                ->,
                red,
                postaction={decorate},
                decoration={
                    markings, mark = at position 0.5 with
                        {\draw[-] (0.1, -0.1) -- (-0.1, 0.1) (0.1, 0.1) -- (-0.1, -0.1);}}},
            -?>/.style = {->, blue},
            q/.style = {font=\small, circle, inner sep=0.05cm, fill=white},
        ]

        \node at (-3, 0) {strip-null};
        \node at (0, 1) {$\gamma$-null};
        \node at (0, -1) {tube-null};
        \node at (3.3, 0) {Lebesgue null};

        \draw[->, transform canvas={yshift=0.15cm}] (-2.2, 0.2) -- (-0.8, 0.8);
        \draw[-x>, transform canvas={yshift=-0.15cm}] (-0.8, 0.8) -- (-2.2, 0.2);
        \draw[->, transform canvas={yshift=0.15cm}] (-2.2, -0.2) -- (-0.8, -0.9);
        \draw[-x>, transform canvas={yshift=-0.15cm}, sloped] (-0.8, -0.9) --
            node[font=\scriptsize, below] {($d \ge 3$)} (-2.2, -0.2);
        \draw[-?>, transform canvas={xshift=-0.12cm}] (0, -0.8) -- node[q] {?} (0, 0.8);
        \draw[-x>, transform canvas={xshift=0.14cm}] (0, 0.8) -- (0, -0.8);
        \draw[-?>, transform canvas={yshift=-0.14cm}] (2.2, 0.2) -- node[q] {?}(0.8, 0.8);
        \draw[->, transform canvas={yshift=0.15cm}] (0.8, 0.8) -- (2.2, 0.2);
        \draw[-x>, transform canvas={yshift=-0.15cm}] (2.2, -0.2) -- (0.8, -0.9);
        \draw[->, transform canvas={yshift=0.15cm}, sloped] (0.8, -0.9) -- (2.2, -0.2);
    \end{tikzpicture}
    \caption{The graph of implications between the nullity properties for $d\ge 2$. A crossed arrow indicates that the implication does not hold in general (however, strip-null and tube-null coincide when $d=2$). We do not know whether the implications marked with question marks hold (see \cref{q_null}).}
    \label{fig:null_graph}
\end{figure}

The graph of implications between the various nullity concepts is shown in \cref{fig:null_graph}. In particular, we will see in \cref{strip_null_is_gstrip_null} that if a set $A\subseteq \R^d$ is strip-null, then $\gamma(A) = 0$. Combined with \cref{gstrip_main}, this yields the following result:
\begin{Cor}
    Let $A\subseteq \R^d$ be a bounded set such that its boundary $\partial A$ is strip-null. Then $A$ is both distance and measure Kolmogorov.
\end{Cor}
For instance, this shows that the Koch snowflake is measure Kolmogorov.

An important special case is when $A$ is a compact Lebesgue null set. In this case, the image is also Lebesgue null. This has the following consequence:
\begin{Cor}
    Suppose that $A\subseteq \R^d$ is a compact strip-null set. Then for every $\varepsilon > 0$, there is a 1-Lipschitz map $f : \R^d \to \R^d$ such that $|f(x) - x| \le \varepsilon$ for every $x\in \R^d$ and $f(A)$ can be covered by finitely many hyperplanes.
\end{Cor}
Since the Sierpiński carpet is a compact tube-null set in the plane by the previously mentioned result in \cite{SierpinskiTN}, it follows that it is distance Kolmogorov. This answers the question in \cite[Remark 4.3]{BEM}.

Finally, in \cref{sec:surface}, we will see some examples of $\gamma$-null sets. We will show that convex hypersurfaces and regular $\mathcal C^2$ hypersurfaces in $\R^d$ are $\gamma$-null. Combined with the $\sigma$-subadditivity of $\gamma$, this yields the following corollary:
\begin{Cor}
    Suppose that the boundary of a bounded set $A\subseteq \R^d$ can be covered by countably many convex or regular $\mathcal C^2$ hypersurfaces. Then $A$ is both distance and measure Kolmogorov.
\end{Cor}

\section{Proof of the equivalence} \label{sec:equiv}

The goal of this section is to prove \cref{dist_to_measure}. We begin by showing the following lemma, which states that a 1-Lipschitz map with small displacement results in small measure changes. The main ideas of the proof were also used in \cite{BEM}.

\begin{Lemma} \label{dist_to_measure_ball}
    Fix some $r > 0$. Suppose that for some $\varepsilon > 0$, $f : \R^d \to \R^d$ is a 1-Lipschitz map with the property $|f(x) - x|\le \varepsilon$ for every $x\in \R^d$. Then for every measurable set $A\subseteq B(0, r)$, $\lambda(A) - \lambda(f(A)) = O(\varepsilon)$ as $\varepsilon \to 0$. The $O(\varepsilon)$ here depends on $d$ and $r$, but not on $f$ or $A$.
\end{Lemma}
\begin{proof}
    Let $0 < \varepsilon < r$. Since $f$ is continuous, a standard topological argument shows that $f(B(0, r)) \supseteq B(0, r - \varepsilon)$. The map $r \mapsto \lambda(B(0, r))$ is differentiable, hence $\lambda(B(0, r)) - \lambda(B(0, r - \varepsilon)) = O(\varepsilon)$.

    Let $A\subseteq B(0, r)$ be an arbitrary measurable set. It follows from the assumption that $f$ is 1-Lipschitz that $\lambda(f(A))\le \lambda(A)$. This also holds for $B(0, r) \setminus A$, therefore,
    \begin{align*}
        \lambda(f(A)) + \lambda(B(0, r) \setminus A) &\ge
        \lambda(f(A)) + \lambda(f(B(0, r) \setminus A)) \\&\ge
        \lambda(f(B(0, r))) \\&\ge
        \lambda(B(0, r - \varepsilon)) \ge
        \lambda(B(0, r)) - O(\varepsilon) \\&=
        \lambda(A) + \lambda(B(0, r) \setminus A) - O(\varepsilon).
    \end{align*}
    We have shown that $\lambda(f(A)) \ge \lambda(A) - O(\varepsilon)$, completing the proof.
\end{proof}

We will also use the following result of Ciosmak \cite{LipExt}:

\begin{Thm}[{\cite[Proposition 3.8 and Theorem 3.9 (iii)\,$\Rightarrow$\,(ii)]{LipExt}}] \label{dist_ext}
    Let $X$ and $Y$ be real Hilbert spaces, and let $v : X \to Y$ be a 1-Lipschitz affine map. Suppose that for some $A\subseteq X$ and $\varepsilon > 0$, $f : A \to Y$ is a 1-Lipschitz map with the property that $\nr{f(x) - v(x)}\le \varepsilon$ for every $x\in A$. Then $f$ has a 1-Lipschitz extension $\tilde f:X\to Y$ such that $\nr{\tilde f(x) - v(x)}\le \varepsilon$ for every $x\in X$.
\end{Thm}

\begin{proof}[Proof of \cref{dist_to_measure}]
    Let $A\subseteq \R^d$ be a measurable distance Kolmogorov set, and let $\varepsilon > 0$. By the boundedness assumption on $A$, we can choose an $r$ such that $A\subseteq B(0, r)$. For every $\delta > 0$, there is a 1-Lipschitz map $f : A \to \R^d$ such that $f(A)$ is a polygon and $|f(x) - x| \le \delta$ for every $x\in A$. Applying \cref{dist_ext} with $v=\id_{\R^d}$, $f$ extends to $\R^d$ as a 1-Lipschitz map such that $|f(x) - x| \le \delta$ for every $x\in \R^d$. It follows from \cref{dist_to_measure_ball} that if $\delta$ is small enough, then $\lambda(f(A)) \ge \lambda(A) - \varepsilon$. Thus we have shown that $A$ is measure Kolmogorov.
\end{proof}

\section{Generalized strips and the proof of \cref{gstrip_main}}
\label{sec:gstrip}

Suppose that the set $A\subseteq \R^2$ is the union of finitely many strips of small total width. We want to map $A$ into the union of finitely many lines with a 1-Lipschitz map $f$ such that $\sup |f(x) - x|$ is small. For a strip $S\subseteq \R^2$, consider the 1-Lipschitz map $f_S$ that projects $S$ onto its halving line and is a translation on both components of $\R^d\setminus S$. One might try to construct $f$ as the composition of maps of the form $f_S$. However, after each step, the remaining strips may break into two pieces. This means that $\sup |f(x) - x|$ is possibly exponentially large in the number of strips. To avoid this problem, we introduce \emph{generalized strips}. Intuitively, generalized strips let us apply the maps $f_S$ simultaneously, preventing the exponential blow-up.

Before we define generalized strips, we need some concepts from convex analysis. See \cite{ConvAnal} for a reference.

\begin{Def}
    The \emph{proximal map} of a convex function $f : \R^d \to \R$ is defined as
    \[\prox_f(x)=\argmin_{y\in \R^d} \zj*{f(y) + \frac{|x - y|^2}2},\]
    where $\argmin_{y} g(y)$ denotes the unique minimizer of $g$.
\end{Def}

\begin{Prop}[{\cite[Proposition 12.28]{ConvAnal}}] \label{prox_contr}
    For every convex function $f$, $\prox_f$ is 1-Lipschitz.
\end{Prop}

\begin{Def}
    Let $f : \R^d \to \R$ be a function, and let $x\in \R^d$. The \emph{subdifferential} of $f$ at $x$ (denoted by $\partial f(x)$) is the set
    \[\partial f(x) = \ha[\big]{v \in \R^d}
    {\forall y \; f(y) \ge f(x) + \sk{v}{y - x}}.\]
\end{Def}

Most of the time, we will consider convex functions on $\R^d$. In this case, the subdifferential is nonempty everywhere (see \cite[Proposition 16.27]{ConvAnal}). Furthermore, the subdifferential $\partial f(x)$ depends only on $f|_U$, where $U$ is an arbitrary neighborhood of $x$. Since this restriction can be extended to a convex function on $\R^d$ if $U$ is small enough, this allows us to consider the subdifferential of a convex function defined only on some open set.

Note that if $f$ is convex and differentiable at $x$, then $\partial f = \{\nabla f(x)\}$, where $\nabla f$ denotes the gradient of $f$.

\begin{Prop}[{\cite[Proposition 16.44]{ConvAnal}}] \label{prox_subdiff}
    For every $x, y\in \R^d$, $y = \prox_f(x)$ if and only if $x - y \in \partial f(y)$.
\end{Prop}

An easy corollary of \cref{prox_subdiff} is the following fact:

\begin{Prop} \label{prox_dist}
    If $f$ is $L$-Lipschitz, then $|\prox_f(x)-x|\le L$ for every $x\in \R^d$.
\end{Prop}

Using the proximal map, we can now define generalized strips.

\begin{Def}
    A function $f:\R^d \to \R$ is \emph{polyhedral} if it is the maximum of finitely many affine functions.
\end{Def}

It is clear that every polyhedral function is convex.

\begin{Def}
    Let $f:\R^d \to \R$ be a polyhedral function. The \emph{generalized strip associated to $f$} (denoted by $S(f)$) is the set of points $x\in \R^d$ such that $f$ is not differentiable at $\prox_f(x)$.

    A set $S\subseteq \R^d$ is called a \emph{generalized strip} if $S=S(f)$ for some polyhedral function $f$. The \emph{width} $w(S)$ of a generalized strip is defined as
    \[w(S)=2\inf\ha{\Lip f}{S = S(f)},\]
    where $\Lip f$ is the Lipschitz constant of $f$.
\end{Def}

See \cref{fig:gstrip} for an example of a generalized strip in the plane.

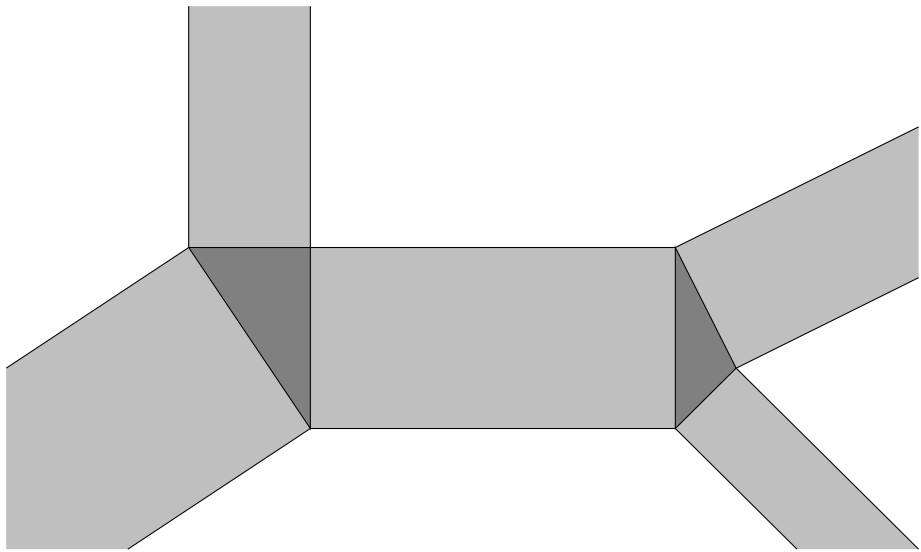
\begin{figure}
    \centering
    \begin{tikzpicture}[scale=0.8]
        \fill[lightgray]
        (-2, 5) -- (-2, 1) -- (4, 1) -- (8, 3) --
        (8, 0.5) -- (5, -1) -- (8, -4) --
        (6, -4) -- (4, -2) -- (-2, -2) -- (-5, -4) --
        (-7, -4) --
        (-7, -1) -- (-4, 1) -- (-4, 5);
    
        \draw[fill=gray]
            (-2, 1) -- (-2, -2) -- (-4, 1) -- cycle;
        \draw[fill=gray]
            (4, 1) -- (5, -1) -- (4, -2) -- cycle;
    
        \draw (-2, 5) -- (-2, 1) -- (4, 1) -- (8, 3);
        \draw (8, 0.5) -- (5, -1) -- (8, -4);
        \draw (6, -4) -- (4, -2) -- (-2, -2) -- (-5, -4);
        \draw (-7, -1) -- (-4, 1) -- (-4, 5);
        
    \end{tikzpicture}
    \caption{The generalized strip $S$ associated to $f(x, y) = \max(y, -2y, x - y - 4, -2x + y - 4)$. Outside $S$, $\prox_f$ is a translation on each connected component. On the rectangular parts of $S$, $\prox_f$ is a projection onto a side combined with a translation. The darker triangles are both mapped to a single point.}
    \label{fig:gstrip}
    \end{figure}

Note that the function $f$ is not unique. For example, if $c$ is a constant, then $S(f)=S(f + c)$.

\begin{Def}
    A \emph{strip} of width $w > 0$ in $\R^d$ is the closed region between two parallel hyperplanes of distance $w$.
\end{Def}

\begin{Prop} \label{strip_is_gstrip}
    A strip $S\subseteq \R^d$ of width $w$ is a generalized strip of width $w$.
\end{Prop}
\begin{proof}
    Let $v$ be a normal vector of length $w/2$ of the hyperplanes bounding $S$. It is easy to check that $S$ can be written as
    \[S=\ha[\big]{x}{|\sk{x}{v} + c|\le |v|^2}\]
    for some constant $c$. Let $f_0(x) = \sk{x}{v} + c$ and $f=|f_0|$. We will show that $S=S(f)$.

    It is clear that the set of points where $f$ is not differentiable is the hyperplane $f_0^{-1}(0)$. Notice that for $x\in f_0^{-1}(0)$, every vector in the subdifferential $\partial f(x)$ is orthogonal to this hyperplane, or equivalently, parallel to $v$. From this, we can check that the subdifferential $\partial f(x)$ is the line segment $[-v, v]$. Therefore, by \cref{prox_subdiff},
    \[S(f)=f_0^{-1}(0) + [-v, v] = f_0^{-1}\zj[\big]{[-|v|^2, |v|^2]} = S.\]
    Since $\Lip f=|v|=w/2$, this also shows that $w(S) \le w$.

    To prove that $w(S)\ge w$, we have to show that $\Lip g\ge w/2$ whenever $S=S(g)$. Let $\varepsilon > 0$. Since $g$ is polyhedral, it is almost everywhere differentiable. Choose an $x$ from the $\varepsilon$-neighborhood of $f_0^{-1}(0)$ such that $g$ is differentiable at $x$. By \cref{prox_subdiff}, $\prox_g(x + \nabla g(x))=x$, hence $x + \nabla g(x)\notin S(g)=S$. Since $B(x, w/2-\varepsilon)\subseteq S$, it follows that $|\nabla g(x)|\ge w/2-\varepsilon$, therefore, $\Lip g\ge w/2 - \varepsilon$. Taking the limit $\varepsilon \to 0$ yields the bound $\Lip g\ge w/2$.
\end{proof}

\begin{Def}
    We define the outer measure $\gamma : \mathcal P(\R^d) \to [0, \infty]$ as
    \[\gamma(A) = \inf \ha*{\sum_{i=1}^\infty w(S_i)}{\text{$S_1, S_2, \dots$ are generalized strips such that $A \subseteq \bigcup_{i=1}^\infty S_i$}}.\]
\end{Def}
The $\sigma$-subadditivity is clear from the definition. We can check that $S(0)=\emptyset$, which shows that the empty set is a generalized strip with $w(\emptyset)=0$. Therefore, $\gamma(\emptyset) = 0$, so it is indeed an outer measure.

Recall that a set in $\R^d$ is \emph{strip-null} if it can be covered by strips of arbitrarily small total width. The following proposition is an easy consequence of \cref{strip_is_gstrip}:

\begin{Prop} \label{strip_null_is_gstrip_null}
    If $A$ is strip-null, then $\gamma(A)=0$.
\end{Prop}

\begin{Prop} \label{gstrip_null_is_null}
    If $\gamma(A) = 0$, then $A$ is Lebesgue null.
\end{Prop}
\begin{proof}
    Since $A$ can be decomposed into countably many bounded pieces, it suffices to consider the case when $A$ is bounded. Choose an $r$ such that $A\subseteq B(0, r)$. Let $\varepsilon > 0$, and choose a sequence of generalized strips $S_1=S(f_1), S_2=S(f_2), \dots$ such that $A \subseteq \bigcup_{n=1}^\infty S_n$ and $\sum_{n=1}^\infty \Lip f_n < \varepsilon$. We can apply \cref{dist_to_measure_ball} to $\prox_{f_n}$ by \cref{prox_contr,prox_dist}: if $\varepsilon$ is small enough, then
    \begin{equation} \label{gstrip_null_is_null.bound}
        \lambda(S_n \cap B(0, r)) - \lambda(\prox_{f_n}(S_n \cap B(0, r))) \le c \Lip f_n
    \end{equation}
    for some constant $c$. Since $f_n$ is differentiable almost everywhere, it follows from the definition of generalized strip that $\lambda(\prox_{f_n}(S_n)) = 0$. Combined with \cref{gstrip_null_is_null.bound}, this yields the bound $\lambda(S_n \cap B(0, r)) \le c \Lip f_n$. Thus we obtain
    \[
        \lambda(A) \le
        \lambda\zj*{\bigcup_{n=1}^\infty (S_n \cap B(0, r))} \le
        \sum_{n=1}^\infty \lambda(S_n \cap B(0, r)) \le
        \sum_{n=1}^\infty c \Lip f_n \le c\, \varepsilon.
    \]
    As $c$ does not depend on $\varepsilon$, taking the limit $\varepsilon \to 0$ shows that $\lambda(A) = 0$.
\end{proof}

\begin{Prop} \label{gstrip_polyhedron}
    Every generalized strip is a polyhedron.
\end{Prop}
\begin{proof}
    Let $f=\max_i f_i$, where the $f_i$ are finitely many affine functions. We will show that $S(f)$ is a polyhedron. Let $v_i$ be the gradient of $f_i$. We may assume that the $v_i$ are pairwise distinct. It suffices to show that each point $x\in \partial(S(f))$ is contained in some hyperplane of the form
    \[H_{i, j, k} = \ha[\big]{x}{(f_i - f_j)(x - v_k)=0}\]
    with $i\ne j$. The set of points where $f$ is not differentiable is closed, so $S(f)$ is also closed. Therefore, $f$ is not differentiable at $y=\prox_f(x)$. This means that $f(y)=f_i(y)=f_j(y)$ for some $i\ne j$.
    
    Suppose for contradiction that $x\notin H_{i, j, k}$ for every $k$. Since $(f_i - f_j)(y) = 0$, this implies that $y\ne x - v_k$, or equivalently, $x - \prox_f(x)\ne v_k$. By continuity, there is a neighborhood $U\ni x$ such that for every $x'\in U$ and $k$, $x' - \prox_f(x')\ne v_k$. If $f$ is differentiable at $y'=\prox_f(x')$, then $f=f_k$ in some neighborhood of $y'$, from which it follows by \cref{prox_subdiff} that $x'-y'=v_k$. Therefore, $x'\in S(f)$ for every $x'\in U$. This contradicts the assumption that $x\in \partial(S(f))$.
\end{proof}

We will now begin proving \cref{gstrip_main}. The idea of the proof is the following: using a compactness argument, we can choose finitely many generalized strips of small total width covering the boundary. Then, we can combine them into a generalized strip of small width, which can be used to construct a 1-Lipschitz map with the required properties. This combination is provided by the following key lemma:

\begin{Lemma} \label{gstrip_union}
    Let $f$ and $g$ be polyhedral functions on $\R^d$. Then there exists a polyhedral function $h$ such that $S(f) \cup S(g) \subseteq S(h)$ and $\Lip h \le \Lip f + \Lip g$.
\end{Lemma}
\begin{proof}
    Let $f=\max_i f_i$ and $g=\max_j g_j$, where the $f_i$ and $g_j$ are finitely many affine functions. Let $v_i$ and $u_j$ be the gradients of $f_i$ and $g_j$, respectively. Define the function
    \[h=\max_{i, j} \zj[\big]{f_i + g_j + \sk{v_i}{u_j}}.\]
    It is clear that $\Lip f=\max_i |v_i|$, $\Lip g = \max_j |u_j|$ and $\Lip h = \max_{i, j} |v_i + u_j|$, from which the property $\Lip h \le \Lip f + \Lip g$ follows.

    We will now check that $S(f)\cup S(g)\subseteq S(h)$. Since the definition of $h$ is symmetric in $f$ and $g$, it suffices to show the containment $S(f)\subseteq S(h)$. Assume that $x\notin S(h)$, and let $y=\prox_h(x)$. By the definition of $S(h)$, $h$ is differentiable at $y$. Since $h$ is polyhedral, this means that there is an open set $U\ni y$ and indices $i$ and $j$ such that
    \begin{equation} \label{gstrip_union.h_eq}
        h(z) = f_i(z) + g_j(z) + \sk{v_i}{u_j}
    \end{equation}
    for every $z\in U$. By the definition of $h$, for every $i'$,
    \begin{equation} \label{gstrip_union.h_bound}
        h(z) \ge f_{i'}(z) + g_j(z) + \sk{v_{i'}}{u_j}.
    \end{equation}
    Using that $\nabla f_i = v_i$ and combining \cref{gstrip_union.h_eq,gstrip_union.h_bound} yields
    \[f_i(z + u_j) = f_i(z) + \sk{v_i}{u_j} \ge
    f_{i'}(z) + \sk{v_{i'}}{u_j} = f_{i'}(z + u_j).\]
    Since $i'$ was arbitrary, it follows that $f|_{U + u_j} = f_i|_{U+u_j}$. Therefore, $f$ is differentiable at $y + u_j$ and $\nabla f(y + u_j)=v_i$.

    From \cref{gstrip_union.h_eq}, we can see that $\nabla h(y)=v_i + u_j$. By \cref{prox_subdiff},
    \[x = y + \nabla h(y) = y + v_i + u_j = y + u_j + \nabla f(y + u_j),\]
    therefore, $\prox_f(x)=y + u_j$. This shows that $x\notin S(f)$.
\end{proof}

Next, we will show that for compact sets, it is enough to consider coverings by a single generalized strip in the definition of $\gamma$:

\begin{Lemma} \label{gstrip_compact_cover}
    For a compact set $C\subseteq \R^d$,
    \[\gamma(C) = \inf\ha*{w(S)}{\text{$S\supseteq C$ is a generalized strip}}.\]
\end{Lemma}
\begin{proof}
    The inequality $\gamma(C) \le \inf w(S)$ is trivial.

    For the other direction, fix an $\varepsilon > 0$ and take generalized strips $S_1, S_2 \dots$ covering $C$ such that $\sum_{i=1}^\infty w(S_i) \le \gamma(C) + \varepsilon$. It follows from \cref{gstrip_polyhedron} that there is a sequence $H_1, H_2, \dots$ of hyperplanes covering $\bigcup_{i=1}^\infty \partial S_i$. Let $S'_j$ be a strip of width $\varepsilon/2^j$ containing $H_j$ in its interior. By \cref{strip_is_gstrip}, $S'_j$ is a generalized strip, furthermore,
    \begin{equation} \label{gstrip_compact_cover.w_bound}
        \sum_{i=1}^\infty w(S_i) + \sum_{j=1}^\infty w(S'_j) =
        \sum_{i=1}^\infty w(S_i) + \sum_{j=1}^\infty \frac{\varepsilon}{2^j} \le
        \gamma(C) + 2\varepsilon.
    \end{equation}
    Notice that
    \[C\subseteq \bigcup_{i=1}^\infty (\Int S_i \cup \partial S_i) \subseteq
    \bigcup_{i=1}^\infty \Int S_i \cup \bigcup_{j=1}^\infty H_i\subseteq
    \bigcup_{i=1}^\infty \Int S_i \cup \bigcup_{j=1}^\infty \Int S'_j.\]
    Using the compactness of $C$, we can choose generalized strips $\tilde S_1, \dots, \tilde S_n$ from $\{S_i\}_i$ and $\{S_j'\}_j$ such that $C\subseteq \bigcup_{k=1}^{n} \Int \tilde S_k$. Furthermore, \cref{gstrip_compact_cover.w_bound} implies that $\sum_{k=1}^n w(\tilde S_k) \le \gamma(C)+2\varepsilon$.

    Choose $f_1, \dots, f_n$ such that $\tilde S_k=S(f_k)$ and $2\sum_{k=1}^n \Lip f_k \le \gamma(C)+3\varepsilon $. By repeatedly applying \cref{gstrip_union}, we obtain an $f$ with the properties $\bigcup_{k=1}^{n} \tilde S_k\subseteq S(f)$ and $\Lip f \le \sum_{k=1}^n \Lip f_k$. We can see that $S = S(f)$ satisfies $C\subseteq S$ and $w(S) \le \gamma(C) + 3\varepsilon$. Since $\varepsilon$ was arbitrary, this completes the proof.
\end{proof}

\begin{Lemma} \label{gstrip_contr}
    Let $A\subseteq \R^d$ be bounded, and suppose that $\partial A\subseteq S(f)$ for some polyhedral function $f$. Then $\prox_f(A)$ is a polyhedron.
\end{Lemma}
\begin{proof}
    It is clear that the set of points where a polyhedral function $f$ is not differentiable can be covered by finitely many hyperplanes. Therefore, it is enough to show that for every $y\in \partial(\prox_f(A))$, $f$ is not differentiable at $y$. Assuming the contrary, there is a neighborhood $U\ni y$ such that $f|_U$ is affine with some gradient $v$.

    By the boundedness assumption, $\overline{A}$ is compact, so it follows from the continuity of $\prox_f$ that $y\in \overline{\prox_f(A)}=\prox_f\zj[\big]{\overline A}$. Choose a $z\in \overline A$ such that $y=\prox_f(z)$. We know that $\nabla f(y)=v$, so $z=y+v$ by \cref{prox_subdiff}. By the definition of $S(f)$, $z\notin S(f)$, and so $z\notin \partial A$, therefore $z\in \Int A$. We can also check that for $x\in U + v$, we have $x = (x - v) + \nabla f(x - v)$, so it follows from \cref{prox_subdiff} that $\prox_f(x) = x - v$. This shows that
    \[y\in U\cap (\Int A - v) = \prox_f((U + v) \cap \Int A)
    \subseteq \prox_f(A),\]
    hence $y\in \Int \prox_f(A)$, contradicting the assumption $y\in \partial(\prox_f(A))$.
\end{proof}

\begin{proof}[Proof of \cref{gstrip_main}]
    Let $A$ be a bounded set with $\gamma$-null boundary. The measurability of $A$ is clear from \cref{gstrip_null_is_null}. It suffices to show that $A$ is distance Kolmogorov, from which it follows by \cref{dist_to_measure} that $A$ is also measure Kolmogorov. By the boundedness assumption, $\partial A$ is compact. \Cref{gstrip_compact_cover} implies that for every $\varepsilon > 0$, there exists a generalized strip $S$ such that $\partial A\subseteq S$ and $w(S) < 2\varepsilon$. Choose an $\varepsilon$-Lipschitz polyhedral function $f$ such that $S=S(f)$. We will now check that $F=\prox_f$ satisfies the conditions for $A$ being distance Kolmogorov. We know that $F$ is 1-Lipschitz by \cref{prox_contr}. Furthermore, by \cref{prox_dist}, $|F(x)-x|\le \varepsilon$ for every $x\in \R^d$. Finally, it follows from \cref{gstrip_contr} that $F(A)$ is a polyhedron.
\end{proof}

\begin{Rem}
    To show that a bounded set with $\gamma$-null boundary is measure Kolmogorov, there is no need for \cref{dist_ext}. Recall that \cref{dist_to_measure} was shown by extending the 1-Lipschitz map defined on $A$ to $\R^d$, then applying \cref{dist_to_measure_ball} to the extended map. However, the construction in the proof of \cref{gstrip_main} already gives a 1-Lipschitz map $F$ such that $|F(x) - x|\le \varepsilon$ for every $x\in \R^d$, which means that the extension can be avoided.
\end{Rem}

\section{Examples of \texorpdfstring{$\gamma$}{γ}-null sets}
\label{sec:surface}

We will first show that the boundary of a convex set is $\gamma$-null.

\begin{Lemma} \label{convex_neighborhood_cover}
    Let $C\subseteq \R^d$ be a bounded convex set. For $r > 0$, let $C_r$ be the (open) $r$-neighborhood of $C$. Then $\gamma(C_r\setminus \Int C) \le 2r$.
\end{Lemma}
\begin{proof}
    We will show that $\gamma(C_r\setminus \Int C)\le 2(r+\varepsilon)$ for every $\varepsilon > 0$. Choose a finite set $D\subseteq C$ such that for every $x\in C$, there is a point $y\in D$ such that $|x - y| \le \varepsilon$. The convex hull $P=\conv(D)\subseteq C$ is a polytope, therefore, it is a polyhedron in the sense that it is the intersection of finitely many half-spaces. Write $P=\ha{x}{f_i(x) \le 0\; \forall i}$, where the $f_i$ are finitely many affine functions. Let $v_i$ be the gradient of $f_i$. Without loss of generality, we may assume that $|v_i| = r+\varepsilon$ for every $i$. Let $f = \max(0, \max_i f_i)$ and $S = S(f)$. It is clear that $w(S) \le 2(r+\varepsilon)$.

    It remains to show that $C_r\setminus \Int C \subseteq S$. Let $x\in C_r\setminus \Int C$, and assume for contradiction that $x\notin S$. Then $f$ is differentiable at $y=\prox_f(x)$, so either $f=0$ or $f=f_i$ in a neighborhood of $y$. If $f=0$ in a neighborhood, then $y\in \Int P$, furthermore, it follows from \cref{prox_subdiff} that $x=y$, contradicting the assumption $x\notin \Int C$.

    If $f=f_i$ in a neighborhood of $y$, then $x=y + v_i$ and $f_i(y) \ge 0$. Since $x \in C_r$, there is a point $z_0\in C$ such that $|x - z_0| < r$. Choose a point $z\in D$ satisfying $|z_0 - z| \le \varepsilon$. By the triangle inequality, $|x - z| < r + \varepsilon = |v_i|$. Note that $z\in P$, hence $f_i(z) \le 0$. Using the fact that $f_i$ is $|v_i|$-Lipschitz, we can see that
    \[0 \le f_i(y) - f_i(z) = f_i(x) - f_i(z) - |v_i|^2 \le |v_i||x - z| - |v_i|^2 < 0,\]
    a contradiction.
\end{proof}

\begin{Thm} \label{convex_gstrip_null}
    Every convex hypersurface in $\R^d$ is $\gamma$-null.
\end{Thm}
\begin{proof}
    By the $\sigma$-subadditivity of $\gamma$, it is enough to show that a convex hypersurface is locally $\gamma$-null. A bounded convex surface can be obtained as a subset of $\partial C$ for some bounded convex set $C$. For every $r > 0$, $\partial C \subseteq C_r \setminus \Int C$, so we are done by \cref{convex_neighborhood_cover}.
\end{proof}

Next, we will show that regular $\mathcal C^2$ hypersurfaces are $\gamma$-null. We will need the following property of the subdifferential:

\begin{Prop}[{\cite[Proposition 16.9]{ConvAnal}}] \label{subdiff_prod}
    Let $f : \R^n \to \R$ and $g : \R^m \to \R$, and let $h : \R^n \times \R^m \to \R$, $h(x, y) = f(x) + g(y)$. Then $h$ has subdifferential
    \[\partial h(x, y) = \partial f(x) \times \partial g(y).\]
\end{Prop}

\begin{Lemma} \label{dc_gstrip_cover}
    Let $g : \R^{d-1} \to \R$ and $h : \R^{d-1} \to \R$ be $1/3$-Lipschitz polyhedral functions, and let $f = g - h$. Then for every $\varepsilon > 0$, there exists a generalized strip $S\subseteq \R^{d}$ such that $w(S) \le 8\varepsilon$ and
    \[|y - f(x)| \le \varepsilon \implies (x, y)\in S\]
    for every $x\in \R^{d-1}$ and $y\in \R$.
\end{Lemma}
\begin{proof}
    Consider the polyhedral function
    \[F(x, y) = 2\varepsilon\max \zj*{2g(x) - y, y + 2h(x)}.\]
    We will show that $S = S(F)$ satisfies the conditions. It is clear that $F$ is $4\varepsilon$-Lipschitz, therefore $w(S)\le 8\varepsilon$.

    Assume that $|y - f(x)| \le \varepsilon$, and let $(\tilde{x}, \tilde{y}) = \prox_F(x, y)$. We have to show that $F$ is not differentiable at $(\tilde{x}, \tilde{y})$. Note that since $\partial_y(2g(x) - y) \ne \partial_y(y+2h(x))$, it is enough to prove the equality $2g(\tilde x) - \tilde y = \tilde y + 2h(\tilde x)$. Assume for contradiction that this equality does not hold.
    
    First, consider the case when $2g(\tilde x) - \tilde y > \tilde y + 2h(\tilde x)$. By continuity, this inequality also holds in a neighborhood, therefore $F(x', y')= 2\varepsilon(2g(x') - y')$ in a neighborhood. Using \cref{subdiff_prod}, we can see that
    \[\partial F(\tilde x, \tilde y) =
    \partial(4\varepsilon g)(\tilde x) \times \partial (- 2 \varepsilon \id)(\tilde y) =
    4\varepsilon \partial g(\tilde x) \times \{-2\varepsilon\}.\]
    If follows from \cref{prox_subdiff} that $x - \tilde x \in 4\varepsilon \partial g(\tilde x)$ and $y = \tilde y - 2\varepsilon$. Since $g$ is $1/3$-Lipschitz, we can see that $|x - \tilde x| \le 4\varepsilon/3\le 3\varepsilon/2$. Note that that $g$ and $h$ are both $1/3$-Lipschitz, therefore $g(x) \ge g(\tilde x) - \varepsilon/2$ and $h(x) \le h(\tilde x) + \varepsilon/2$. This implies that
    \[f(x) - y \ge \zj*{g(\tilde x) - \frac{\varepsilon}2} - \zj*{h(\tilde x) + \frac{\varepsilon}{2}} - (\tilde y - 2\varepsilon) =
    \frac{(2g(\tilde x) - \tilde y) - (\tilde y + 2 h(\tilde x))}2 + \varepsilon > \varepsilon,\]
    contradicting the initial assumption.

    Now assume that $2g(\tilde x) - \tilde y < \tilde y + 2h(\tilde x)$. A similar calculation shows that $|x - \tilde x|\le 3\varepsilon/2$ and $y=\tilde y + 2\varepsilon$. It follows that
    \[y - f(x) \ge (\tilde y + 2\varepsilon) - \zj*{g(\tilde x) + \frac{\varepsilon}2} + \zj*{h(\tilde x) - \frac{\varepsilon}{2}} =
    \frac{(\tilde y + 2 h(\tilde x)) - (2g(\tilde x) - \tilde y)}2 + \varepsilon > \varepsilon,\]
    which is, again, a contradiction.
\end{proof}

\begin{Lemma} \label{convex_approx}
    Let $U\subseteq \R^d$ be a bounded convex open set, and let $f : U \to \R$ be an $L$-Lipschitz convex function. Then for every $\varepsilon > 0$, there exists an $L$-Lipschitz polyhedral function $g : \R^d \to \R$ such that $|f(x) - g(x)| \le \varepsilon$ for every $x\in U$.
\end{Lemma}
\begin{proof}
    Let $L$ be the Lipschitz constant of $f$. We can choose a finite set of points $\{x_1, \dots, x_n\}\subseteq U$ such that for every $x\in U$, there is an $i$ such that $|x - x_i| \le \varepsilon/2L$. For each $x_i$, choose some $v_i\in \partial f(x_i)$. Since $U$ is open, it is easy to see that $|v_i|\le L$. Now take
    \[g(x) = \max_i \zj[\big]{\sk{v_i}{x - x_i} + f(x_i)}.\]
    It is clear from the definition of subdifferential that $g \le f$. Furthermore, $g$ is $L$-Lipschitz and $g(x_i) = f(x_i)$ for every $i$.

    Let $x\in U$, and choose an $i$ such that $|x - x_i| \le \varepsilon/2L$. We can check that
    \[f(x) \le f(x_i) + L|x - x_i| = g(x_i) + L|x - x_i| \le g(x) + 2L|x - x_i| \le g(x) + \varepsilon.\]
    This shows that $|f(x) - g(x)|\le \varepsilon$.
\end{proof}

We will also need the following theorem of Aleksandrov \cite{Alexandrov}. For completeness, we include the proof, as we use a slightly different formulation:
\begin{Prop} \label{convex_diff}
    Let $U\subseteq \R^d$ be a convex open set, and let $f : U \to \R$ be differentiable. Suppose that $\nabla f$ is $M$-Lipschitz. Then the map $g(x) = f(x) + M|x|^2/2$ is convex.
\end{Prop}
\begin{proof}
    Let $x\in U$, and consider the line $\ell(t) = x + t v$ for some $v\in \R^d$. We have to show that $g\circ \ell$ is convex, for which it suffices to show that $(g\circ \ell)'$ is non-decreasing. A simple calculation shows that
    \[(g \circ \ell)'(t) = M \sk{\ell(t)}{v} + \sk{\nabla f(\ell(t))}{v}.\]
    For $t < s$, we can see that
    \begin{align*}
        (g\circ \ell)'(s) - (g\circ \ell)'(t) &=
        M \sk{\ell(s) - \ell(t)}{v} + \sk{\nabla f(\ell(s)) - \nabla f(\ell(t))}{v} \\&\ge
        M \sk{(s - t) v}{v} - |\nabla f(\ell(s)) - \nabla f(\ell(t))| |v| \\&\ge
        M (s - t)|v|^2 - M|\ell(s) - \ell(t)| |v| = 0,
    \end{align*}
    which shows that $(g \circ \ell)'$ is indeed non-decreasing.
\end{proof}

\begin{Thm} \label{surface_gstrip_null}
    Let $U\subseteq \R^{d-1}$ be an open set, and suppose that $F \in \mathcal C^2(U, \R^d)$ is a regular hypersurface, that is, $F'$ has rank $d-1$ everywhere. Then $\gamma(F(U)) = 0$.
\end{Thm}
\begin{proof}
    It is a well-known that a regular $\mathcal C^2$ hypersurface is locally the graph of a $\mathcal C^2$ function in an appropriate coordinate system. Formally, for every $x\in U$, there exist open sets $x\in V\subseteq U$ and $0\in W\subseteq \R^{d-1}$, a diffeomorphism $\varphi : W \to V$, an orthogonal map $A : \R^d \to \R^d$ and a function $f \in \mathcal C^2(W, \R)$ such that $(F \circ \varphi)(x) = A \circ (x, f(x))$ for $x\in W$. Furthermore, $f$ can be chosen such that $\nabla f(0) = 0$.

    It is clear from the definitions that $\gamma$-null sets are preserved under isometries of $\R^d$. Therefore, we may assume that $A$ is the identity. By the $\sigma$-subadditivity of $\gamma$, it is enough to show that the theorem holds locally. Note that $\nabla f$ is $\mathcal C^1$, so it is locally Lipschitz. Without loss of generality, we may assume that $W$ is bounded and convex, furthermore, $\nabla f$ is $M$-Lipschitz on $W$ for some $M > 0$. Since $\nabla f(0) = 0$ and $\nabla f$ is continuous, we may also assume that $|\nabla f| \le 1/6$ on $W$. We will further assume that $W\subseteq B(0, 1/(6M))$.

    Let $h(x) = M |x|^2/2$. By \cref{convex_diff}, the function $g = f + h$ is convex. Since $|\nabla h(x)| = |M x|\le 1/6$ on $W$, it is clear that $g$ and $h$ are both $1/3$-Lipschitz. Using \cref{convex_approx}, we can choose $1/3$-Lipschitz polyhedral functions $\tilde g$ and $\tilde h$ such that $|g - \tilde g|\le \varepsilon$ and $|h - \tilde h| \le \varepsilon$ on $W$ for an arbitrary $\varepsilon > 0$. Let $\tilde f = \tilde g - \tilde h$. Note that $|f - \tilde f|\le 2\varepsilon$. By \cref{dc_gstrip_cover}, there is a generalized strip $S$ such that $\graph f \subseteq S$ and $w(S) \le 16\varepsilon$. Since $\varepsilon$ was arbitrary, this shows that $\gamma(\graph f) = 0$.
\end{proof}

Finally, we will give an example of a set that is $\gamma$-null but not tube-null. In particular, this shows that for $d = 2$, $\gamma$-null is strictly weaker than tube-null.

\begin{Prop} \label{pinned_dist_set_bound}
    For every $A \subseteq \R^d$,
    \[\gamma(A) \le 2\lambda^*\zj[\big]{\ha[\big]{|x|}{x \in A}},\]
    where $\lambda^*$ is the Lebesgue outer measure.
\end{Prop}
\begin{proof}
    Let $\varepsilon > 0$, and let $D=\ha[\big]{|x|}{x \in A}$. Choose a sequence of open intervals $I_1, I_2, \dots \subseteq (0, \infty)$ such that $D\subseteq \bigcup_{i=1}^\infty I_i$ and for $r_i = \lambda(I_1)$, $\sum_{i=1}^\infty r_i \le \lambda^*(D) + \varepsilon$. Consider the annulus 
    \[A_i = \ha[\big]{x}{|x|\in I_i}= B(0, \sup I_i) \setminus \overline B(0, \inf I_i).\]
    Since $B(0, \sup I_i)$ is the open $r_i$-neighborhood of $\overline B(0, \inf I_i)$, \cref{convex_neighborhood_cover} implies that $\gamma(A_i) \le 2r_i$. It follows from the subadditivity of $\gamma$ that
    \[\gamma(A) \le \sum_{i=1}^\infty \gamma(A_i) \le 2 \sum_{i=1}^\infty r_i \le 2 (\lambda^*(D) + \varepsilon).\]
    Since $\varepsilon$ was arbitrary, this shows that $\gamma(A)\le 2\lambda^*(D)$.
\end{proof}

\begin{Cor} \label{nontn_gstrip_null}
    For $d \ge 2$, there is a set in $\R^d$ that is $\gamma$-null but not tube-null.
\end{Cor}
\begin{proof}
    We will use an example of a non--tube-null set given in \cite{CSV}.  Let $F\subseteq [1/2, 1]$ be a Lebesgue null set with Hausdorff dimension larger than 1/2, and let
    \[E=\ha[\big]{x\in \R^d}{|x|\in F}.\]
    It was shown in \cite[Proposition 6]{CSV} that $E$ is not tube-null. On the other hand, it follows from \cref{pinned_dist_set_bound} that $\gamma(E)\le 2\lambda(F)=0$.
\end{proof}

\section{Open questions}

Our main question is the following:

\begin{Question} \label{q_null}
    Let $A\subseteq \R^d$ be a Lebesgue null set for some $d \ge 2$. Is it true that $\gamma(A) = 0$?
\end{Question}

We have seen in \cref{gstrip_null_is_null} that every $\gamma$-null set is Lebesgue null. For $d = 1$, the converse is also true, and for $d \ge 2$, we do not know any counterexample. If the answer to \cref{q_null} is positive, then as a consequence of \cref{gstrip_main}, every bounded Jordan measurable set is both distance and measure Kolmogorov.

It is clear from \cref{strip_null_is_gstrip_null} that when $d=2$, a counterexample to \cref{q_null} cannot be tube-null. We have seen in \cref{nontn_gstrip_null} that the construction in \cite{CSV} is not a counterexample. Another non-trivial example of a non--tube-null set was given by Shmerkin and Suomala \cite{ShmerkinSuomala} using fractal percolation. This construction even has the stronger property that it cannot be covered by neighborhoods of convex curves of arbitrarily small total width, therefore, we cannot use \cref{convex_neighborhood_cover} to show that it is $\gamma$-null. Thus, our next question is as follows:

\begin{Question} \label{q_jordan_measurable}
    Is the fractal percolation construction of a non--tube-null set by Shmerkin and Suomala $\gamma$-null?
\end{Question}

Our last question is motivated by \cref{surface_gstrip_null}. Note that $\R^{d-1}\subseteq \R^d$ is $\gamma$-null. Thus, \cref{surface_gstrip_null} says that the image of a specific $\gamma$-null set under a sufficiently smooth map $\R^d \to \R^d$ is $\gamma$-null. The question is whether this holds more generally:

\begin{Question} \label{q_image_gstrip_null}
    Let $d \ge 2$, and let $f : \R^d \to \R^d$ be a $\mathcal C^\infty$ diffeomorphism. Suppose that $\gamma(A) = 0$ for some $A\subseteq \R^d$. Is it true that $\gamma(f(A)) = 0$?
\end{Question}

\section*{Acknowledgements}
The author would like to thank his supervisors M. Elekes and T. Keleti for their guidance and many helpful discussions. The author is also grateful to V. Kaluža for his comments on the proof of \cref{dist_to_measure}.

\printbibliography

\end{document}